\documentclass[11pt]{article}

\usepackage{amssymb, amsmath, amsthm, graphicx}
\usepackage[left=1in,top=1in,right=1in]{geometry}

\date{}

\theoremstyle{plain}
      \newtheorem{theorem}{Theorem}[section]
      \newtheorem{lemma}[theorem]{Lemma}
            
      \newtheorem{problem}[theorem]{Problem}
      
      \newtheorem{corollary}[theorem]{Corollary}
      \newtheorem{proposition}[theorem]{Proposition}
      \newtheorem{conjecture}[theorem]{Conjecture}
\theoremstyle{definition}

\theoremstyle{remark}

\title{Quasiplanar graphs, string graphs, and the Erd\H os-Gallai problem}
\author{Jacob Fox\thanks{Stanford University, Stanford, CA. Supported by a Packard Fellowship and by NSF award DMS-1855635. Email: {\tt jacobfox@stanford.edu.}} \and J\'anos Pach\thanks{Supported by NKFIH grants K-131529, Austrian Science Fund Z 342-N31, and ERC Advanced Grant 882971``GeoScape.''Email:
{\tt pach@cims.nyu.edu}.}\and  Andrew Suk\thanks{Department of Mathematics, University of California at San Diego, La Jolla, CA, 92093 USA. Supported an NSF CAREER award and NSF award DMS-1952786. Email: {\tt asuk@ucsd.edu}.} }

\begin{document}

\maketitle

\begin{abstract}

An \emph{$r$-quasiplanar graph} is a graph drawn in the plane with no $r$ pairwise crossing edges. Let $s \geq 3$ be an integer and $r=2^s$. We prove that there is a constant $C$ such that every $r$-quasiplanar graph with $n \geq r$ vertices has at most $n\left(Cs^{-1}\log n\right)^{2s-4}$ edges.

A graph whose vertices are continuous curves in the plane, two being connected by an edge if and only if they intersect, is called a \emph{string graph}. We show that for every $\epsilon>0$, there exists $\delta>0$ such that every string graph with $n$ vertices whose chromatic number is at least $n^{\epsilon}$ contains a clique of size at least $n^{\delta}$. A clique of this size or a coloring using fewer than $n^{\epsilon}$ colors can be found by a polynomial time algorithm in terms of the size of the geometric representation of the set of strings.

In the process, we use, generalize, and strengthen previous results of Lee, Tomon, and others. All of our theorems are related to geometric variants of the following classical graph-theoretic problem of Erd\H os, Gallai, and Rogers. Given a $K_r$-free graph on $n$ vertices and an integer $s<r$, at least how many vertices can we find such that the subgraph induced by them is $K_s$-free?

\end{abstract}

\section{Introduction}

A \emph{topological graph} is a graph drawn in the plane with points as vertices and edges as continuous curves connecting
some pairs of vertices. The curves are allowed to cross, but cannot pass through vertices other than their
endpoints.  If the edges are drawn as straight-line segments, then
the graph is \emph{geometric}.  If no $r$ edges in a topological graph $G$ are pairwise crossing, then $G$ is called {\em $r$-quasiplanar}.

The following is a longstanding unsolved problem in the theory of topological graphs; see, \emph{e.g.,}~\cite{BrMP05}.

\begin{conjecture}\label{quasi}
The number of edges of every $r$-quasiplanar graph of $n$ vertices is $O_r(n)$.
\end{conjecture}

Conjecture~\ref{quasi} has been proved for $r\le 4$. See~\cite{AgAP97,AcT07,Ac09}.

\smallskip
The {\em intersection graph} of a family of geometric objects, $\mathcal{S}$, is a graph with vertex set $\mathcal{S}$, in which two vertices are joined by an edge if and only if their intersection is nonempty. If $\cal S$ consists of continuous curves (or line segments) in the plane, then their intersection graph is called a \emph{string graph} (resp., a \emph{segment graph}).
\smallskip

A natural approach to prove Conjecture~\ref{quasi} is the following. Removing a small disc around every vertex of an $r$-quasiplanar graph $G$, we are left with a family of continuous curves $\mathcal{S}$ in the plane, no $r$ of which are pairwise crossing. These curves define a $K_r$-free string graph $H$. Suppose that the chromatic number of $H$ satisfies $\chi(H)\le f(r)$. Then each color class corresponds to the edges of a planar subgraph of $G$. Thus, the size of each color class is at most $3n-6$, provided that $n \geq 3$. This would immediately imply that every $r$-quasiplanar graph with $n$ vertices has at most $(3n-6)f(r)=O_r(n)$ edges, as required.

Surprisingly, this approach is not viable. In 2014, Pawlik, Kozik, Krawczyk, Laso\'n, Micek, Trotter, and Walczak~\cite{PaKK14} represented a class of $K_3$-free graphs originally constructed by Burling~\cite{Bu65} as {segment graphs} whose chromatic numbers can be arbitrarily large. Shortly after, Walczak~\cite{Wa15} strengthened this result by proving that there are $K_3$-free segment graphs on $n$ vertices in which every independent set is of size $O(\frac{n}{\log\log n})$.

Using the same approach, in order to prove Conjecture~\ref{quasi} for some $r$, it would be sufficient to show that there is a constant $g(r)$ with the property that the vertex set of every $K_r$-free string graph can be colored by $g(r)$ colors such that each (string) graph induced by one of the color classes is $K_4$-free. Indeed, the result of Ackerman~\cite{Ac09} cited above implies that the number of edges in each color class is $O(n)$. Interestingly, this approach is also not viable. Krawczyk and Walczak (see Theorem 1.2 in \cite{KW}) constructed $K_r$-free string graphs on $n$ vertices that require $\Omega_r(\log\log n)$ colors in any coloring of the vertices such that each color classes induces a $K_{r-1}$-free string graph.  However, the following question remains open.
\smallskip

\begin{problem}\label{mainproblem}
Fix an integer $r\geq 4$. Is it true that every $K_r$-free segment graph on $n$ vertices has an induced subgraph on $\Omega_r(n)$ vertices which is $K_{r-1}$-free?
\end{problem}

Building upon the work of McGuinness \cite{Mc00}, Suk \cite{Su14} showed that every $K_r$-free \emph{segment} graph on $n$ vertices has a $K_{r-1}$-free induced subgraph with at least $\Omega_r(\frac{n}{\log n})$ vertices. (See also  \cite{RW19a,RW19b}.) For \emph{string} graphs, in general, until now the best known result, due to Fox and Pach~\cite{FP14}, was weaker: they could only guarantee the existence of an independent set and, hence, a $K_{r-1}$-free induced subgraph, of size at least $\frac{n}{(\log n)^{O(\log r)}}$.
\smallskip

In a different range, where $r$ grows polynomially in $n$, Tomon \cite{To20} solved a longstanding open problem by showing that there is a constant $c'>0$ such that every string graph on $n$ vertices has a clique or an independent set of size  $n^{c'}$.

Our next theorem slightly strengthens the result of Fox and Pach~\cite{FP14}.
\begin{theorem}\label{stringind}
Let $s$ be a positive integer and $r=2^s$. Every $K_{r}$-free string graph on $n \geq r$ vertices has an independent set of size at least $n(cs/\log n)^{2s-2},$ where $c>0$ is an absolute constant.
\end{theorem}


At the beginning of Section~\ref{sec5}, we show how to deduce from Theorem~\ref{stringind} the following strengthening of Tomon's above mentioned theorem~\cite{To20}.

\begin{corollary}\label{epsdeltastring}
For any $\epsilon>0$, there is $\delta>0$ such that every string graph $G$ on $n$ vertices has a clique of size at least $n^{\delta}$ or its chromatic number is at most $n^{\epsilon}$. (In the latter case, obviously, $G$ has an independent set of size at least $n^{1-\epsilon}$.)
\end{corollary}

Theorem~\ref{stringind} guarantees the existence of a large \emph{independent} set in a $K_r$-free string graph $G$. If, in the spirit of Problem~\ref{mainproblem}, we want to find only a large \emph{$K_{r-1}$-free} induced subgraph in $G$, we can do better.

\begin{theorem}\label{string1}
For any $n \geq r \geq 3$, every $K_r$-free string graph with $n$ vertices has a $K_{r-1}$-free induced subgraph with at least $c\frac{n}{\log^2 n}$ vertices, where $c>0$ is an absolute constant.

\end{theorem}

At the expense of another logarithmic factor, we can also find an induced subgraph with no clique of size $\lceil{r}/{2}\rceil$.

\begin{theorem}\label{string2}
For any $n \geq r \geq 3$, every $K_r$-free string graph with $n$ vertices has a $K_{\lceil r/2\rceil}$-free induced subgraph with at least $c\frac{n}{\log^3 n}$ vertices, where $c>0$ is an absolute constant.
\end{theorem}

Now we return to the original motivation behind our present note: to estimate from above the number of edges of an $r$-quasiplanar \emph{topological} graph of $n$ vertices. As mentioned before, for $r\le 4$, Conjecture~\ref{quasi} is true. For any  $r\ge 5$, the best previously known upper bounds were $n(\log n)^{O(\log r)}$ and $O(n(\log n)^{4r-16})$, established in \cite{FP14} and \cite{PRT06}, respectively. For \emph{geometric} graphs, for any fixed $r \geq 5$, Valtr~\cite{Va98} obtained the upper bound $O(n\log n)$.  See \cite{ack} for a survey.

Using the result of Ackerman \cite{Ac09} as the base case of an induction argument, and exploiting several properties of string graphs established by Lee~\cite{Le17}, Tomon~\cite{To20}, and Fox and Pach~\cite{FP12b,FP14}, we will deduce the following improved upper bound for the number of edges of $r$-quasiplanar topological graphs.


\begin{theorem}
\label{secondcor}
Let $s\geq 3$ be an integer and $r=2^s$. Then every $r$-quasiplanar graph with $n \geq r$ vertices has at most $n(\frac{c\log n}{s})^{2s-4}$ edges, where $c>0$ is an absolute constant.
\end{theorem}

Setting $s=3$, for instance, we obtain that every $8$-quasiplanar topological graph on $n$ vertices has $O(n(\log n)^2)$ edges, which is better than the previously known bound of $O(n(\log n)^{16})$ \cite{Ac09,PRT06}. For $r=\delta \log n$, Theorem~\ref{secondcor} immediately implies the following.

\begin{corollary}
\label{thirdcor}
For any $\epsilon>0$ there is $\delta>0$ such that every topological graph with $n$ vertices and at least $3n^{1+\epsilon}$ edges has $n^{\delta}$ pairwise crossing edges.
\end{corollary}

The factor $3$ in front of the term $n^{1+\epsilon}$ guarantees that the graph is not planar. Otherwise, we could not even guarantee that there is one crossing pair of edges.
\smallskip

In the special case where the strings are allowed to cross only a bounded number of times, some results very similar to Theorems~\ref{stringind} and~\ref{secondcor} were established in \cite{FP12}.
\smallskip

 Theorems~\ref{stringind}, \ref{string1}, \ref{string2}, and Corollary~\ref{epsdeltastring} guarantee the existence of an independent set or a $K_p$-free induced subgraph for some $p>2$ in a string graph satisfying certain conditions. All of these sets and subgraphs can be found by \emph{efficient polynomial time algorithms} in terms of the size of a geometric representation of the underlying string graph. For example, the proof of Corollary \ref{epsdeltastring} yields the following algorithmic result.

\begin{proposition}\label{alg}
For any $\epsilon>0$ there is $\delta>0$ with the following property.
Given a representation of a string graph on $n$ vertices as an intersection graph of strings, there is a polynomial time algorithm which either properly colors the vertices with $n^{\epsilon}$ colors or finds a clique of size $n^{\delta}$.
\end{proposition}

Erd\H{o}s and Gallai~\cite{ErG61} raised the following question. Given a pair of integers, $2 \leq p < r$, how large of a $K_p$-free induced subgraph must be contained in every $K_r$-free graph of $n$ vertices? For $p=2$, we obtain Ramsey's problem: how large of an independent set must be contained in every $K_r$-free graph of $n$ vertices. The special case $p=r-1$ was considered by Erd\H os and Rogers~\cite{ErR62}. These problems have since been extensively studied. For many striking results, see, \emph{e.g.,} \cite{DuR11,DRR14,GoJ20,Kr94,Kr95,Su05,Wo13}.
Apart from our last two results listed in the introduction, all statements in this paper can be regarded as geometric variants of the Erd\H os-Gallai-Rogers problem for string graphs.
\smallskip

The rest of this note is organized as follows. In Section~\ref{sec3}, we apply the analogues of the separator theorem and the K\H ov\'ari-S\'os-Tur\'an theorem for string graphs \cite{Le17,FP14} to establish Theorems~\ref{string1} and~\ref{string2}. In Section~\ref{sec4}, we present a simple technical lemma (Lemma~\ref{largeanddense}) and some of its consequences needed for the proof of Theorems~\ref{stringind} and~\ref{secondcor}. The proofs of these two theorems and Corollary~\ref{epsdeltastring} are given in Section~\ref{sec5}. The last section contains some concluding remarks.

\smallskip

Throughout this paper, $\log$ always stands for the binary logarithm. The letters $c$ and $C$ appearing in different theorems denote unrelated positive constants. Whenever they are not important, we will simply omit floor and ceiling signs.

\section{Separators---Proofs of Theorems~\ref{string1} and~\ref{string2}}\label{sec3}

In this section, we prove Theorems  \ref{string1} and \ref{string2}.
We need the separator theorem for string graphs, due to Lee \cite{Le17}.  A \emph{separator} in a graph $G = (V, E)$ is a subset $S$ of the vertex set $V$ such that no connected component of $G \setminus S$ has more than $\frac{2}{3}|V|$ vertices. Equivalently, $S$ is a separator of $G$ if there is a partition $V = S \cup V_1 \cup V_2$ with $|V_1|, |V_2| \leq \frac{2}{3}|V|$ such that no vertex in $V_1$ is adjacent to any vertex in $V_2$.

\begin{lemma}[\cite{Le17}]\label{lee}
Every string graph with $m$ edges has a separator of size at most $c_1\sqrt{m}$, where $c_1$ is an absolute constant.
\end{lemma}

We now prove the following theorem which immediately implies Theorem \ref{string1}.  Let us remark that the \emph{neighborhood} of vertex $v$ does not include $v$.

\begin{theorem}\label{isolated}
There is an absolute constant $c> 0$ with the following property.  Every string graph $G$ on $n>1$ vertices contains an induced subgraph $G'$ on $c\frac{n}{\log^2 n}$ vertices whose every connected component is contained in the neighborhood of a vertex or is an isolated vertex.
\end{theorem}

\begin{proof}
Let $c > 0$ be a sufficiently small constant to be specified later.  We proceed by induction on $n$.  The base case when $n = 2$ is trivial.  For the inductive step, assume that the statement holds for all $n' < n$.  Let $G = (V,E)$ be an $n$-vertex string graph.

If $G$ contains a vertex $v$ of degree at least $cn/\log^2n$, then we are done by setting $G'$ to be the neighborhood of $v$.  Otherwise, we know that there are at most $cn^2/\log^2n$ edges in $G$. By Lemma \ref{lee}, $G$ has a separator $S\subset V$ of size at most $c_1\sqrt{c}n/\log n$, where $c_1$ is the absolute constant from Lemma \ref{lee}.  Hence, there is a partition $V = S\cup V_1\cup V_2$ with $|V_1|, |V_2| \leq \frac{2}{3}|V|$ such that no vertex in $V_1$ is adjacent to any vertex in $V_2$, and $|S| \leq c_1\sqrt{c}n/\log n$.  By applying induction on $V_1$ and $V_2$ and setting $c < \left(\frac{\log 3/2}{c_1}\right)^2$, we obtain an induced subgraph $G'$ on at least

\[c\frac{|V_1|}{\log^2|V_1|} + c\frac{|V_2|}{\log^2|V_2|} \geq c\frac{|V_1| + |V_2|}{\log^2(2n/3)} \geq c\frac{n - \frac{c_1\sqrt{c}n}{\log n}}{\log^2(2n/3)}\]

\[ = c\frac{n}{\log^2 n}\cdot
\frac{1 - \frac{c_1\sqrt{c}}{\log n}}{\left(1 - \frac{\log(3/2)}{\log n}\right)^2}
\geq c\frac{n}{\log^2 n}\]

\noindent vertices such that each component of $G'$ is contained in the neighborhood of a vertex or is an isolated vertex. \end{proof}

To see that Theorem~\ref{isolated} implies Theorem~\ref{string1}, it is sufficient to notice that if $G'$ has a clique of size $r-1$, then  $G$ has a clique of size $r$.
\smallskip

The proof of Theorem \ref{string2} is very similar to that of Theorem~\ref{string1}. Here, we need the following analogue of the K\H ov\'ari-S\'os-Tur\'an theorem, which can also be deduced from Lemma \ref{lee} (see Conjecture 3.3 in \cite{FP14}).

\begin{lemma}[\cite{Le17,FP14}]\label{bipartite}

Every $K_{t,t}$-free string graph on $n$ vertices has at most $c_2t(\log t)n$ edges, where $c_2$ is an absolute constant.
\end{lemma}

\medskip

\noindent \emph{Proof of Theorem \ref{string2}.}\;\;
Let $c>0$ be a sufficiently small constant to be determined later.  We proceed by induction on $n$.  The base case $n = 3$ is trivial.  For the inductive step, assume that the statement holds for all $n' < n$.  Let $G = (V,E)$ be a $K_r$-free string graph on $n$ vertices, and let $c_1$ and $c_2$ be the constants from Lemmas~\ref{lee} and~\ref{bipartite}, respectively.

If $G$ has at least $cc_2\frac{n^2}{\log^2 n}$ edges, then, by Lemma \ref{bipartite}, $G$ contains a complete bipartite graph $K_{t,t}$, where $t\geq c\frac{n}{\log^3n}$.  Since $G$ is $K_r$-free, one of these parts must be $K_{\lceil r/2\rceil}$-free, and we are done.

Otherwise, if $G$ has fewer than $cc_2\frac{n^2}{\log^2 n}$ edges, then, by Lemma \ref{lee}, there is a partition $V = S\cup V_1\cup V_2$ with $|V_1|, |V_2| \leq \frac{2}{3}|V|$ such that no vertex in $V_1$ is adjacent to any vertex in $V_2$, and $|S| \leq c_1\sqrt{cc_2}n/\log n$.  If $|V_1| < 3$ or $|V_2| < 3$, then we are done by setting $c$ sufficiently small.  Otherwise, applying the induction hypothesis to $V_1$ and $V_2$, and setting $c < \frac{\log^2(3/2)}{c_2c_1^2}$, we obtain a $K_{\lceil r/2\rceil}$-free induced subgraph $G'\subseteq G$ with at least

\[c\frac{|V_1|}{\log^3|V_1|} + c\frac{|V_2|}{\log^3|V_2|} \geq c\frac{|V_1| + |V_2|}{\log^3(2n/3)} \geq c\frac{n - \frac{c_1\sqrt{cc_2}n}{\log n}}{\log^3(2n/3)}\]

\[ = c\frac{n}{\log^3 n}\cdot
\frac{1 - \frac{c_1\sqrt{cc_2}}{\log n}}{\left(1 - \frac{\log(3/2)}{\log n}\right)^3}
\geq c\frac{n}{\log^3 n} \]

\noindent vertices. \hfill $\Box$ 

\medskip 

\section{A technical lemma for string graphs}\label{sec4}

The \emph{average degree} in a graph $G=(V,E)$ is $d=\frac{2|E|}{|V|}$.
The \emph{edge density} of $G$ is defined as $\frac{|E|}{{|V|\choose 2}}=\frac{d}{|V|-1}$. We say that a graph is \emph{dense} if its edge density is larger than some positive constant (which we will conveniently specify for our purposes).
\smallskip

Using Lee's separator theorem for string graphs (Lemma \ref{lee}), it is easy to deduce the following technical lemma which states that every string graph $G$ has a dense induced subgraph $G'$ whose average degree is not much smaller than the average degree in $G$.

\begin{lemma}\label{largeanddense}
For any $\epsilon>0$, there is $C=C(\epsilon)$ with the following property. Every string graph $G=(V,E)$ with average degree $d=2|E|/|V|$ has an induced subgraph $G[V']$ with average degree $d' \geq (1-\epsilon)d$ and $|V'| \leq Cd'$.
\end{lemma}
\begin{proof}
Let $G=(V,E)$ be a string graph with average degree $d$. We recursively define a nested sequence of induced subgraphs $G_0 \supset G_1 \supset \cdots$.

We begin with $G_0=G$, and let $V_0=V$, $E_0=E$ and $d_0=d$. After obtaining $G_i=(V_i,E_i)$ with $E_i=E(G[V_i])$ and with average degree $d_i=2|E_i|/|V_i|$, we show that $G_i$ is the desired induced subgraph if $d_i \geq |V_i|/C$. Otherwise if $d_i < |V_i|/C$, we have $|E_i| \leq |V_i|^2/(2C)$, and by Lemma \ref{lee}, there is a partition $V_i=U_0 \cup U_1 \cup U_2$ with $|U_1|,|U_2| \leq 2|V_i|/3$,  \[|U_0| \leq c_1\sqrt{|E_i|} \leq c_1\frac{|V_i|}{\sqrt{2C}} \leq  |V_i|/12,\] and there are no edges from $U_1$ to $U_2$.  The last inequality above follows from the fact that $C \geq (12c_1)^2$. We can assume this as we can choose $C$ as large as we want.

We take $G_{i+1}$ to be the induced subgraph on whichever of $G[U_1 \cup U_0]$ and $G[U_2 \cup U_0]$ has larger average degree. As all edges of $G_i$ are in at least one of these two induced subgraphs and $|U_1\cup U_0|+|U_2 \cup U_0|=|V_i|+|U_0|$, the average degree of $G_{i+1}$ satisfies

\[d_{i+1} \geq d_i\frac{|V_i|}{|V_i|+|U_0|}=d_i\frac{1}{1+|U_0|/|V_i|}\]

\[ \geq d_i\frac{1}{1+c_1\sqrt{|E_i|}/|V_i|}   \geq d_i\frac{1}{1+c_1\sqrt{d_i/(2|V_i|)}}.\]

\noindent As $d_i < |V_i|/C$ and $C$ can be chosen sufficiently large, the above inequality implies that $d_{i+1} \geq \frac{9}{10}d_i$.
The inequality $|U_0| \leq |V_i|/12$ implies that $|V_{i+1}| \leq \frac{3}{4}|V_i|$. These two inequalities together imply that $d_{i+1}/|V_{i+1}| \geq \frac{6}{5}d_i/|V_i|$. It follows from the inequality above that 

\[d_{i+1} =d\prod_{j=0}^i d_{j+1}/d_j \geq d\prod_{j=0}^i \frac{1}{1+c_1\sqrt{d_j/(2|V_j|)}} \geq de^{-\sum_{j=0}^i c_1\sqrt{d_j/(2|V_j|)}},\]
where the last inequality uses that $\frac{1}{1+x} \geq e^{-x}$ for any $x>0$.
The sum in the exponent is dominated by a geometric series with common ratio $\sqrt{6/5}>1$, and its largest summand is at most $c_1(1/(2C))^{1/2}$, as $d_i \leq |V_i|/C$. Hence, the sum in the exponent is $O(C^{-1/2})$. Taking $C$ large enough, we have that $d_{i+1} \geq (1-\epsilon)d$ for every $i$ for which $d_{i+1}$ is defined. (We can choose $C=O(1/\epsilon^{2})$ to satisfy this.) Further, as $|V_{i+1}| \leq \frac{3}{4}|V_i| <|V_i|$ for every $i$ for which $V_{i+1}$ is defined, after at most $|V|$ iterations, the above process will terminate with the desired induced subgraph $G_i$.
\end{proof}


The main result of \cite{FP12b} is that every dense string graph contains a dense spanning subgraph which is an incomparability graph. Applying Lemma \ref{largeanddense} to this spanning subgraph with $\epsilon=1/2$, we obtain the following corollary.

\begin{corollary}\label{stringposet}
There is a constant $c>0$ with the following property. Every string graph with $n$ vertices and $m$ edges has a subgraph with at least $c\frac{m}{n}$ vertices which is an incomparability graph with edge density at least $c$. \end{corollary}

Given a graph $G=(V,E)$ and two disjoint subsets of vertices $X, Y\in V$, we say that $X$ is \emph{complete to} $Y$ if $xy\in E$ for all $x\in X$ and $y\in Y$.

The next lemma can be deduced by combining Corollary \ref{stringposet} above with the proof of Theorem~4 in \cite{To20}.  The proof of Theorem 4 in \cite{To20} actually proves a stronger result than stated, where $(n/|X_i|)^c < t$ can be replaced by $c'(n/|X_i|)^{1/2} < t$ in the statement, where $c'$ is an absolute constant.

\begin{lemma}\label{tomondeduction}
There is a constant $c>0$ with the following property. If $G=(V,E)$ is a string graph with $n$ vertices and at least $\alpha n^2$ edges, for some $\alpha>0$, then there are disjoint vertex subsets $X_1,\ldots,X_t\subset V$ for some $t\ge 2$ such that
\begin{enumerate}
\item $X_i$ is complete to $X_j$ for all $i \neq j$, and
\item $|X_i| \geq c \alpha \frac{n}{t^2}$ for every $i$.
\end{enumerate}
\end{lemma}

\section{Back to quasiplanar graphs---Proofs of Theorems~\ref{stringind} and~\ref{secondcor}}\label{sec5}

Before turning to the proof of Theorem~\ref{stringind}, we show how it implies Corollary~\ref{epsdeltastring}.

\medskip 

\noindent \emph{Proof of Corollary~\ref{epsdeltastring}.}\;\;
The most natural technique for properly coloring a graph is by successively extracting maximum independent sets from it. Using this greedy method and the bound in Theorem \ref{stringind}, for $r=2^s$, we obtain a proper coloring of any $K_{r}$-free string graph on $n$ vertices with at most $(\frac{\log n}{cs})^{2s-2}\log n$ colors. Indeed, each time we extract a maximum independent set, the fraction of remaining vertices is at most $1-\alpha$ with $\alpha=(\frac{cs}{\log n})^{2s-2}$. As $1-\alpha<e^{-\alpha}$, after at most $\frac{\log n}{\alpha}$ iterations, no vertex remains.

For a given $\epsilon>0$, choose a sufficiently small $\delta>0$ such that
\begin{enumerate}
\item $2\delta\log\frac{1}{{c\delta}}<\frac{\epsilon}{2}$ and
\item $\log n<n^{\epsilon/2}$ provided that $n^{\delta} \geq 2$.
\end{enumerate}

\noindent Consider any $K_{n^{\delta}}$-free string graph $G$ on $n$ vertices. If $n^{\delta}<2$, then $G$ has no edges and, hence, its chromatic number is $1\leq n^{\epsilon}$. Otherwise, substituting $s=\delta \log n$, Theorem~\ref{stringind} yields that the chromatic number of $G$ is at most 

\[n^{2\delta\log{\frac{1}{c\delta}}}\log n<n^{\epsilon}.\]
\hfill $\Box$ 

\medskip

Now we turn to the proof of Theorem \ref{stringind} which gives, for $r=2^s$, a lower bound on the independence number of a $K_{r}$-free string graph on $n$ vertices.

\medskip 

\noindent \emph{Proof of Theorem \ref{stringind}.}\;\;
Let $c> 0$ be a sufficiently small constant that will be determined later.  Our proof is by double induction on $s$ and $n$. Throughout we let $r=2^s$. The base cases are when $s=1$ (in which case we get an independent set of size $n$) or $n \leq 10$ (in which case we set $c > 0$ to be sufficiently small and get an independent set of size $1$) and are trivial.  Note that $n \geq r$.  The induction hypothesis is that the theorem holds for all $s'<s$ and all $n'$, and for $s'=s$ and all $n'<n$.  Note that we may assume that $r \leq n/4$, as otherwise the theorem holds.
Let $\alpha=c'(\frac{s}{\log n})^2$, where $c'>0$ is a sufficiently small absolute constant. Let $G$ be a $K_{r}$-free string graph on $n$ vertices.

If $G$ has at most $\alpha n^2$ edges, applying Lemma \ref{lee}, there is a vertex partition $V=V_0 \cup V_1 \cup V_2$ with $|V_0| \leq c_1\alpha^{1/2}n$, $|V_1|,|V_2| \leq 2n/3$, and there are no edges from $V_1$ to $V_2$. Note that $|V_0| \leq n/12$ so $|V_1|,|V_2| \geq n/4$.
We obtain a large independent set in $G$ by taking the union of large independent sets in $V_1$ and $V_2$. Using the induction hypothesis applied to $G[V_1]$ and $G[V_2]$, we obtain an independent set in $G$ of order at least

\[|V_1|\left(\frac{cs}{\log |V_1|}\right)^{2s-2}+|V_2|\left(\frac{cs}{\log |V_2|}\right)^{2s-2} \geq (|V_1|+|V_2|)\left(\frac{c s}{\log (2n/3)}\right)^{2s-2}.\]

Note that $|V_1|+|V_2|=n-|V_0| \geq n(1-c_1c'^{1/2}\cdot\frac{s}{\log n})$. We also have

\[\left(\log(2n/3)\right)^{2s-2}=(\log n)^{2s-2}\left(1-\frac{\log (3/2)}{\log n}\right)^{2s-2} \leq
(\log n)^{2s-2}\left(1-\frac{s}{2\log n}\right),\]

\noindent where the last inequality follows from the Binomial Theorem, and using the fact that $s \geq 2$ and $n > 10$.  Substituting in these estimates and using $c'>0$ is sufficiently small, we obtain an independent set of the desired size.

Suppose next that $G$ has at least $\alpha n^2$ edges. By Lemma \ref{tomondeduction}, there is an integer $t\geq 2$ and disjoint vertex subsets $X_1,\ldots,X_t$ such that $X_i$ is complete to $X_j$ for all $i \neq j$ and $|X_i| \geq 4c'' \alpha n/t^{2}$ for $i=1,\ldots,t$ where $0<4c''<1$ is an absolute constant. Losing a factor at most $2$ in the number of sets $X_i$, we may assume $t=2^p$ for a positive integer $p<s$, which implies  $|X_i| \geq c'' \alpha n/t^{2}$. As $G$ is $K_{2^s}$-free, one of these sets $X_i$ induces a subgraph which is  $K_{2^{s-p}}$-free. Let $n_0=|X_i|$. Applying the induction hypothesis to $G[X_i]$, we obtain an independent set of size at least

\[n_0\left(\frac{c(s-p)}{\log n_0}\right)^{2(s-p)-2} \geq c''c'\left(\frac{s}{\log n}\right)^2 n2^{-2p}\left(\frac{c(s-p)}{\log n_0}\right)^{2(s-p)-2}\]\[\geq n\left(\frac{cs}{\log n}\right)^{2s-2}.\]

The last inequality holds, because after substituting $\log n_0 \leq \log n$, the ratio of the right-hand side and the expression in the middle  reduces to

\[\frac{(2c)^{2p}}{c''c'}\left(\frac{s}{\log n}\right)^{2(p-1)}\left(1+\frac{p}{s-p}\right)^{2(s-p)-2} \leq \frac{(2ec)^{2p}}{c''c'}\left(\frac{s}{\log n}\right)^{2(p-1)} \leq 1.\]

At the first inequality, we used $1+x \leq e^x$ with $x=\frac{p}{s-p}$. As for the second inequality, we know that $s \leq \log n$, and we are free to choose the constant $c>0$ as small as we wish (for instance, $c = c''c'/30$  will suffice). This completes the proof.
\hfill $\Box$ 

\medskip 

A careful inspection of the proof of Theorem \ref{stringind} shows that it recursively constructs an independent set of the desired size in a $K_{r}$-free string graph in polynomial time in terms of the size of the geometric representation of the set of strings. Indeed, the proof itself is essentially algorithmic. In the first case, when the string graph is relatively sparse, we apply Lee's separator theorem for string graphs, and take the union of large independent sets from the string graph of the two remaining large vertex subsets after deleting the small separator. In the second case,  when the string graph is relatively dense, we apply Lemma \ref{tomondeduction} to get in the string graph a complete multipartite subgraph with large parts, and we can find a large independent set in one of the parts. However, this does require checking that results from several earlier papers each yield desirable structures in string graphs and incomparability graphs in polynomial time. These results include Lee's separator theorem for string graphs \cite{Le17}, the Fox-Pach result that every dense string graph contains a dense spanning subgraph which is an incomparability graph \cite{FP12b}, and some extremal results of Tomon \cite{To20} for incomparability graphs.

A set of vertices $X\subseteq V$ in a graph $G=(V,E)$ is said to be \emph{$r$-independent} if it does not induce a clique of size $r$, that is, if $G[X]$ is $K_r$-free. In particular, a $2$-independent set is simply an independent set. Note that the proof of Theorem~\ref{stringind} carries through to the following generalization concerning the Erd\H{o}s-Gallai problem for string graphs.

\begin{theorem}\label{EGstring}
Let $s,q$ be positive integers with $s>q$. Every $K_{2^s}$-free string graph $G$ on $n \geq 2^s$ vertices contains a $2^q$-independent set of size at least $$\frac{1}{2^{2s}}\left(\frac{c(s+1-q)}{\log n}\right)^{2s-2q}n,$$
\noindent where $c>0$ is an absolute constant.
\end{theorem}

\begin{proof} (Sketch) We follow the proof of Theorem \ref{stringind}, making minor modifications. The proof is by double induction on $s$ and $n$, with the base cases $s=q$ or $n=2^s$ being trivial. We let $\alpha=c'\left(\frac{s+1-q}{\log n}\right)^2$. As in the proof of Theorem \ref{stringind}, if $G$ has at most $\alpha n^2$ edges, we apply the string graph separator lemma (Lemma \ref{lee}). We delete the separator, use the induction hypothesis on the resulting components, and take the union of the $2^q$-independent sets in the components to get a $2^q$-independent set of the desired size in $G$.
\smallskip

So, we may assume $G$ has more than $\alpha n^2$ edges. By Lemma \ref{tomondeduction}, there is an integer $t\geq 2$ and disjoint vertex subsets $X_1,\ldots,X_t$ such that $X_i$ is complete to $X_j$ for all $i \neq j$ and $|X_i| \geq 4c'' \alpha n/t^{2}$ for $i=1,\ldots,t$, where $c''>0$ is an absolute constant. Losing a factor at most $2$ in the number of sets $X_i$, we may assume that $t=2^p$ for a positive integer $p<s$, which implies $|X_i| \geq c'' \alpha n/t^{2}$. As $G$ is $K_{2^s}$-free, at least one of the sets $X_i$ induces a subgraph which is  $K_{2^{s-p}}$-free.
\smallskip

The proof now splits into two cases, depending on whether $s-p > q$ or not. If $s-p>q$, the rest of the proof goes through as in the proof of Theorem \ref{stringind}. If $s-p \leq q$, then $X_i$ is the desired $2^q$-independent set. Indeed, we have

\begin{eqnarray*} |X_i| & \geq &   c'' \alpha n/t^{2} \geq c''\alpha n/2^{2s} = c''c'\frac{1}{2^{2s}}\left(\frac{s+1-q}{\log n}\right)^2 n \geq \frac{1}{2^{2s}}\left(\frac{c(s+1-q)}{\log n}\right)^2 n  \\ & \geq &  \frac{1}{2^{2s}}\left(\frac{c(s+1-q)}{\log n}\right)^{2s-2q} n,
\end{eqnarray*}

\noindent for a sufficiently small absolute constant $c>0$, as desired.
\end{proof}
\smallskip

We complete the section by proving Theorem \ref{secondcor}, which gives an upper bound on the number of edges of a $r$-quasiplanar graph with $n$ vertices for $r=2^s$ a perfect power of two.

\medskip 

\noindent \emph{Proof of Theorem \ref{secondcor}.}\;\;
Let $s \geq 3$ be an integer and $r=2^s$. We have to prove that every $r$-quasiplanar graph on $n \geq r$ vertices has at most $n(C\frac{\log n}{s})^{2s-4}$ edges, where $C$ is an absolute constant.

For any $r$-quasiplanar graph $G=(V,E)$ on $n$ vertices, delete a small disk around each vertex and consider the string graph whose vertex set consists of the (truncated) curves in $E$. As $G$ is $r$-quasiplanar, the resulting string graph is $K_{2^s}$-free.

Applying Theorem \ref{EGstring} with $q=2$, we obtain a subset $E' \subset E$ with

\[|E'| \geq \frac{1}{2^{2s}}\left(\frac{c(s-1)}{\log |E|}\right)^{2s-4}|E| \geq \left(\frac{c'(s-1)}{\log n}\right)^{2s-4}|E|,\] 

\noindent for some absolute constants $c,c'>0$ such that $G'=(V,E')$ is $4$-quasiplanar. According to Ackerman's result \cite{Ac09}, every $4$-quasiplanar graph on $n$ vertices has at most a linear number of edges in $n$, that is, we have $|E'| \leq C'n$ for a suitable constant $C'>0$. Putting these two bounds together, we get the desired upper bound 

\[|E| \leq C'n\left(\frac{\log n}{c'(s-1)}\right)^{2s-4} \leq n\left(C\frac{\log n}{s}\right)^{2s - 4},\]

\noindent provided that $C$ is sufficiently large.

\hfill $\Box$ 

\medskip

\section{Concluding remarks}\label{sec6}

\textbf{A.} A family of graphs $\mathcal{G}$ is said to be \emph{hereditary} if for any $G\in\mathcal{G}$, all induced subgraphs of $G$ also belong to $\mathcal{G}$. Obviously, the family of string graphs is hereditary.

The proof of Lemma \ref{largeanddense} only uses the fact that there is a separator theorem for string graphs. A careful inspection of the proof shows that the same result holds, instead of string graphs, for any hereditary family of graphs $\mathcal{G}$ such that every $G=(V,E)\in\mathcal{G}$ has a separator of size $O(|E|^{\alpha}|V|^{1-2\alpha}),$ for a suitable constant $\alpha=\alpha(\mathcal{G})>0$.

Similar techniques were used in \cite{FP08,FP10,FP12,LRT79}. Our Lemma \ref{largeanddense} enables us to simplify some of the proofs in these papers.
\medskip





\noindent \textbf{B.} \emph{Circle graphs} are intersection graphs of chords of a circle.  Gy\'arf\'as \cite{Gy85} proved that every circle graph with clique number $r$ has chromatic number at most $O(r^24^r)$.  Kostochka \cite{Ko88}, and Kostochka and Kratochv\'il \cite{KK97} improved this bound to $O(r^22^r)$ and $O(2^r)$ respectively. Recently, a breakthrough was made by Davies and McCarty \cite{DM21}, who obtained the upper  bound $O(r^2)$. Shortly after this, Davies \cite{Da21} further improved this bound to $O(r\log r)$, which is asymptotically best possible due to a construction of Kostochka \cite{Ko88}. By taking the union of the $r-2$ largest color classes in a proper coloring with the minimum number of colors, Davies' result implies that every circle graph on $n$ vertices with clique number $r$ contains an induced subgraph on $\Omega(n/\log r)$ vertices that is $K_{r-1}$-free. We conjecture that this ``naive'' bound can be improved as follows.

\begin{conjecture}
Every $K_r$-free circle graph on $n$ vertices contains an induced subgraph on $\Omega(n)$ vertices which is $K_{r-1}$-free.
\end{conjecture}
\medskip

\noindent \textbf{C.}  The same problem can be raised for the intersection graph $G=G(\mathcal{F})$ of any family $\mathcal{F}$ of $n$ axis-parallel boxes in the plane or in ${\mathbf{R}}^d$ for $d>2$. The clique number of $G$ is equal to the maximum number of times, $r$, that a point is covered by members of $\mathcal{F}$. In the plane, improving an old result of Asplund and Gr\"unbaum~\cite{AsG60}, Chalermsook and Walczak~\cite{ChW21} showed that the chromatic number of $G=G(\mathcal{F})$ is $O(r\log r)$. Just like above, this implies that $G$ has an induced subgraph on $\Omega(n/\log r)$ vertices which is $K_{r-1}$-free. Again, one  can conjecture that this bound can be improved to $\Omega(n)$. It is easy to verify this conjecture for squares. More generally, the statement is true for the intersection graph of any family $\mathcal{F}$ of finitely many ``fat'' convex bodies in ${\mathbf{R}}^d$ (i.e., when there is an absolute upper bound for the ratios of the circumradii to the inradii of the members of $\mathcal{F}$); see~\cite{Pa80}.    

\medskip
\noindent\textbf{Acknowledgement.} We are grateful to Zach Hunter for carefully reading our manuscript and pointing out several mistakes. We also thank an anonymous referee for noticing an inaccuracy in the proof of Theorem~\ref{EGstring}.

\end{document}